\documentclass{amsart}
\usepackage{amsthm}
\usepackage{amsfonts}
\usepackage[latin1]{inputenc}
\newtheorem{thm}{Theorem}%[section]

\newtheorem{lem}[thm]{Lemma}

\theoremstyle{definition}

\newtheorem{defn}[thm]{Definizione}
\theoremstyle{remark}

\newcommand{\R}{\mathbb{R}}

\begin{document}

\title{Finiteness of  rank invariants of multidimensional persistent homology groups}
\author{Francesca Cagliari}
\address{ Dipartimento di Matematica, Universit\`a di Bologna, P.zza di Porta S. Donato
5, I-$40126$ Bologna, Italia}
%\curraddr{Department of Mathematics and Statistics,
%Case Western Reserve University, Cleveland, Ohio 43403}
\email{cagliari@dm.unibo.it}
%%    \thanks will become a 1st page footnote.

\author{Claudia Landi}
\address{Dipartimento
di Scienze e Metodi dell'Ingegneria, Universit\`a di Modena e
Reggio Emilia, Via Amendola 2, Pad. Morselli, I-42100 Reggio
Emilia, Italia} \email{clandi@unimore.it}
\thanks{Research  partially carried out within the activities of ARCES (Universit\`a di Bologna).}

%    General info

\date{} %and, in revised form, June 22, 2001.}



\begin{abstract}
Rank invariants are a parametrized version of Betti numbers of a space multi-filtered by a continuous vector-valued function. In this note we give a sufficient condition for their finiteness. This condition is sharp for spaces embeddable in $\R^n$.
\end{abstract}

\maketitle

\section*{Introduction}
Persistent Topology is a theory for studying objects related to computer vision and computer graphics, which involves analyzing the qualitative and quantitative behavior of real-valued functions defined over topological spaces. More precisely, it studies   the sequence of nested lower level sets of the considered functions and encodes the scale at which  a topological feature  (e.g., a connected component, a tunnel, a void) is created, and when it is annihilated  along this filtration. In this framework, multidimensional  persistent homology groups capture the homology of a multi-parameter increasing family of spaces.  For application purposes, these groups  are further encoded by simply considering their rank, yielding  to a parametrized version of Betti numbers, called rank invariants \cite{CaZo09}.

The aim of this note  is to prove the following theorem, providing a sufficient condition in order that multidimensional persistent homology groups are finitely generated (or, equivalently, rank invariants are finite):

\begin{thm}\label{finitezza_ranghi}
If $X$ is a  compact, locally contractible  subspace of $\R^n$, and $\vec f=(f_1,f_2,\ldots ,f_k):X\rightarrow \R^k$ is a continuous function, then,  for any  $q\in \mathbb{Z}$, the rank invariant $\rho_{(X,\vec f),q}$ takes only finite values: for any $\vec u=(u_1,u_2,\ldots ,u_k)$ and $\vec v=(v_1,v_2,\ldots,v_k)$ in $\R^k$, with $u_j<v_j$ ($j=1,2,\ldots, k$),
 $$\rho_{(X,\vec f),q}(\vec u,\vec v)\stackrel{\mathrm {def}}{=}\mathrm{rank\, }\mathrm{im\,}H_q(X_{\vec f\le \vec u}\hookrightarrow X_{\vec f\le \vec v})<+\infty ,$$  the map being the inclusion, and $X_{\vec f\le \vec u}$ and $X_{\vec f\le \vec v}$ denoting the lower level sets $\{x\in X: f_j(x)\le u_j,\, j=1,2,\ldots , k\}$ and  $\{x\in X: f_j(x)\le v_j,\, j=1,2,\ldots , k\}$, respectively.
\end{thm} 

The finiteness condition for rank invariants has proved to be crucial for the stability of persistence diagrams \cite{CoEdHa07}, and the stability of multidimensional persistent homology groups \cite{CaDiFe07,CeDi*09}. In each of these papers, finiteness of persistent homology ranks was generally  guaranteed by requiring the topological space to be triangulable or imposing a tameness condition on the filtering function. In \cite{ChCo*09} it is  argued  that this functional setting is not large enough to address the problems encountered in practical applications, and stability of persistence diagrams is revisited. The basic assumption to prove stability in \cite{ChCo*09} is the finiteness of persistent homology ranks, but the question of how achieving this is  left unanswered.   Nevertheless, it is known that the ranks of  $0$th persistent homology groups (i.e. size functions) are finite provided that the space is only compact and locally connected \cite{CaDiFe07}. Theorem~\ref{finitezza_ranghi} settles this issue also for multidimensional persistent homology groups of spaces embeddable in the Euclidean space.

\section{Proof of Theorem \ref{finitezza_ranghi}}
We begin recalling the definition of  ENR and the criteria for a space to be an ENR, following \cite{Bredon93}.

\begin{defn}\label{def:ENR}
A topological subspace $Y$ of $\R^n$ is called a {\em Euclidean neighborhood retract} (ENR) if there is  a neighborhood in $\R^n$ of which $Y$ is a retract.
\end{defn}

\begin{thm}[cf. \cite{Bredon93}]\label{thm:ENR}
If $Y\subseteq \R^n$ is locally compact and locally contractible then $Y$ is an ENR. 
\end{thm}

In order to prove Theorem \ref{finitezza_ranghi} we anticipate a lemma.

\begin{lem}\label{inclusione}
Let $X$ be a compact, locally contractible subspace of $\R^n$. Let $P,Q$ be subspaces of $X$ with $\mathrm{cl}_X(Q)\subseteq \mathrm{int}_X(P)$.  Then, for any $q\in \mathbb{Z}$, it holds that 
$$\mathrm{rank\ }\mathrm{im\,}H_q( Q \hookrightarrow P)<+\infty,$$
the map being the inclusion.
\end{lem}

\begin{proof}
We take $L_1=\mathrm{cl}_X(Q)$, $L_2=\mathrm{int}_X(P)$. It is sufficient to show that,  for any $q\in \mathbb{Z}$,  the rank of $\mathrm {im\,}H_q(\iota: L_1 \hookrightarrow L_2)$ is finite, $\iota$ being the inclusion map. To this aim we observe that $L_1$ is closed in $X$, and hence compact. Moreover, $L_2$, being open in $X\subseteq \R^n$, is locally compact and locally contractible, and therefore, by Theorem~\ref{thm:ENR}, $L_2$ is an ENR. Thus, by Definition~\ref{def:ENR}, there exists an open neighborhood $L_3$ of $L_2$ in $\R^n$ and a retraction $r:L_3\rightarrow L_2$. Since $L_3$ is open in $\R^n$, for any $x\in L_3$ there is an open $n$-dimensional cube $Q(x)$ centered at $x$, whose closure is contained in $L_3$. Let us consider the open cover of $L_1$ given by $\mathcal{Q}=\{Q(x)\cap L_1\}_{x\in L_3}$. By compactness, $L_1$ admits a finite subcover: $L_1=\bigcup_{s=1}^r Q(x_s)\cap L_1$. Let us set $K= \bigcup_{s=1}^r \mathrm{cl}_{\R^n}\left(Q(x_s)\right)$. Clearly $L_1$ is contained in $K$, and $K$ is contained in $L_3$. Since $K$ is a finite union of compact polyhedra, it is a polyhedron itself (cf. \cite{RoSa72}), and its homology groups are finitely generated. 

Let us now consider the inclusion $j: L_1\hookrightarrow K$,  and  the restriction of the retraction $r$ to  $K$, $r_{|K}:K\rightarrow L_2$. Since $r$ is a retraction onto $L_2\supseteq L_1$, it holds that, for every $x\in L_1$, $r_{|K}(x)=x$. Hence $ \iota= r_{|K}\circ j:  L_1\hookrightarrow K\rightarrow L_2$. Therefore 
$$\mathrm{rank\, }\mathrm{im\,} H_q(\iota)=\mathrm{rank\,}\mathrm{im\,} \left(H_q(r_{|K})\circ H_q(j)\right)\le \mathrm{rank\, }\mathrm{im\,} H_q(j)\le \mathrm{rank\, } H_q(K)<+\infty.$$
\end{proof}

\begin{proof}[Proof of Theorem \ref{finitezza_ranghi}]

It is sufficient to apply Lemma \ref{inclusione}, for every $\vec u=(u_1,u_2,\ldots ,u_k)$, $\vec v=(v_1,v_2,\ldots,v_k)$ in $\R^k$, with $u_j<v_j$ ($j=1,2,\ldots, k$), setting $Q=X_{\vec f\le \vec u}$ and $P=X_{\vec f\le \vec v}$.  
\end{proof}
 
We conclude this note observing that, for a space $X$ embeddable in $\R^n$, the conditions of compactness and local contractibility  cannot be weakened in order to guarantee finiteness for rank invariants.
Indeed, taking the closed  topologist's sine curve
$$X=\{(x,y)\in\R^2: x\in (0,1],\, y=\sin(1/x)\}\cup \{(x,y)\in\R^2: x=0, \, y\in[-1,1]\},$$ we have an example of a compact but not locally contractible subspace of $\R^2$  whose $0$-th rank invariant, when $X$ is filtered using the height function $f(x,y)=y$, is unbounded. Analogously, taking $X=\{(x,y)\in\R^2: x\in (0,1],\, y=\sin(1/x)\}$, we have an example of a  locally contractible but non-compact subspace of $\R^2$  whose $0$-th rank invariant, when $X$ is filtered using the height function, is unbounded.   

\bibliographystyle{amsplain}
\bibliography{biblio}
\end{document}